\documentclass{ifacmtg}
\usepackage{amsfonts,latexsym,bbm}
\usepackage{natbib}
\usepackage[dvips]{graphicx}

\newcommand\CC{\mathbbm{C}}
\newcommand\UU{\mathbbm{U}}

\newcommand\cD{\mathcal{D}}

\newcommand\cL{\mathcal{L}}

\newcommand\cR{\mathcal{R}}

\newcommand\cW{\mathcal{W}}

\newcommand\tU{\tilde{U}}
\newcommand\tV{\tilde{V}}
\newcommand\tW{\tilde{W}}

\newcommand\alphaU{\alpha_u} 
\newcommand\alphaV{\alpha_v} 
\newcommand\betaU{\beta_u} 
\newcommand\betaV{\beta_v} 
\newcommand\MU{M_u} 
\newcommand\MV{M_v} 
\newcommand\kU{K_u} 
\newcommand\kV{K_v} 
\newcommand\U{{\cal U}} 
\newcommand\V{{\cal V}}
\newenvironment{proof}
   {\noindent {\bf Proof.}}{\hfill$\Box$}
\newenvironment{block}{\left[ \begin{array}}{\end{array}\right] }
\newtheorem{theorem}{Theorem}
\newtheorem{proposition}{Proposition}
\newtheorem{lemma}{Lemma}
\newtheorem{corollary}{Corollary}

\begin{document}
\begin{frontmatter}
\title{Nudelman interpolation, parametrizations of lossless functions 
 and balanced realizations.} 
\author[1]{Jean-Paul Marmorat}
\author[2]{Martine Olivi}
\address[1]{CMA, BP 93, 
06902 Sophia-Antipolis Cedex, FRANCE, 
marmorat@sophia.inria.fr, phone: 33 4 92 38 79 56,
fax: 33 4 92 38 79 98}
\address[2]{INRIA, BP 93, 
06902 Sophia-Antipolis Cedex, FRANCE, 
olivi@sophia.inria.fr, phone: 33 4 92 38 78 77, 
fax: 33 4 92 38 78 58}

\begin{abstract}
 We investigate the parametrization issue for
 discrete-time stable all-pass multivariable systems by means of a  Schur
 algorithm involving a Nudelman interpolation condition. A recursive 
construction of balanced realizations is associated with it, that possesses a 
very good numerical behavior. Several atlases of charts or families of local
parametrizations are presented and for each  atlas a  chart selection
strategy is  proposed.  The last one can be viewed as  a nice mutual encoding
property of lossless functions and turns out to be very efficient. These parametrizations 
allow  for solving optimization problems  within the fields of system
identification and optimal control.
\end{abstract}

\end{frontmatter}

\section{Introduction}
Lossless or stable allpass transfer functions play an important role in
system theory mainly due to the Douglas-Shapiro-Shields factorization: any
proper transfer function can be written as the product of a 
lossless function, which includes the dynamics of the system, and an unstable
factor. In many  problems in which a criterion must be optimized over a set of 
functions, the unstable factor of the optimum can be computed from
the lossless one.  This can be done in  rational $L^2$ approximation 
(\cite{F-O98}), system identification (\cite{Bruls}), multi-objective control
(\cite{Scherer00}). These problems can thus be handled by optimization methods
over the class of lossless functions of prescribed degree, or possibly a
specified subclass.  
It is with such applications in mind that we will address the
parametrization issue.

An interesting and unusual approach of this optimization problems is to use 
 the manifold structure of the class of
lossless 
functions of fixed McMillan degree (\cite{ABG})  and
parameters coming from an atlas of charts. An atlas of charts attached with a manifold
is a collection of local coordinate maps (the charts), whose domains cover
the manifold and such that the changes of coordinates are smooth. 
Using such parameters allows to exactly describe the set on which an optimum
is searched. This  ensures that the optimum will be stable and of the
prescribed order. In practice, a search
algorithm can be run through the manifold as a whole, using a local
coordinate map to describe it locally and changing from one
coordinate map to another when necessary. 

In the literature, atlases of charts have been derived 
both from the state-space approach using nice selections and from the
functional approach using interpolation theory and Schur type algorithms. 
A connection between these two approaches  was found in the scalar (or
SISO) case
(\cite{H-P98}) and generalized  to the matrix case 
(\cite{PHO}). In this paper, an atlas is described in which 
balanced realizations can be computed from the Schur parameters.
The computation involves a product of unitary matrices and thus presents a
nice numerical behavior. Moreover, for some particular choices of the
interpolation points and directions, the balanced realizations possess  a
triangular structure which relates to nice selections (\cite{PHO_SSSC04}). 

The natural framework for these studies is that of complex functions. 
However, systems are often real-valued and their transfer
functions $T$ are 
real,  that is, they satisfy the relation $\overline{T(z)}=T(\bar z)$.
Even if the complex case includes the real
case by restriction, a specific treatment is  actually relevant and was the
initial motivation for this work which notably improves \cite{CDC03}.
In rational $L^2$ approximation for example, 
a real function  may have a complex best approximant.
This is the case for  the  function 
$f(z)=1/z^3-1/z$ which  admits three minima: a real  one and two complex
ones, which  achieve  the best relative error.

In this paper, atlases are constructed in which 
lossless functions are represented by balanced realizations built
recursively from interpolation data as in \cite{PHO}. But  instead of the Nevanlinna-Pick
interpolation problem used there we consider here the more
general  Nudelman interpolation problem. 
This very general framework allows to construct several atlases, including
that of  \cite{PHO}, and to describe the subclass of real
functions.  For each particular atlas presented in this
work, we propose  
a simple method to find  an "adapted chart"
for a given lossless function.  This last point, together with their nice
numerical behavior, make   
these parametrizations  an interesting tool for solving the optimization
problems mentioned before.

\section{Preliminaries.}
This paper is concerned with finite dimensional, stable,
discrete-time systems and their transfer functions which happen to be 
rational functions analytic outside the closed unit disk. 
Interpolation theory usually deals with functions that are analytic 
in the open
 unit disk.  To 
relate  these  two situations, we  use the transformation $F\to
F^\sharp$ defined by
\begin{equation} 
\label{Rsharp} 
F^\sharp(z)= F^*(1/ z),~~~F^*(z)=F(\bar z)^*.
\end{equation} 

Let 
\[J=\left[\begin{array}{cc}I_p & 0\\ 0 & - I_p\end{array}\right].\]
A $2p \times 2p$ rational 
matrix function $\Theta(z)$ is called $J$-lossless (or conjugate $J$-inner)
 if, at every point of analyticity $z$ of $\Theta(z)$ it satisfies
\begin{eqnarray}
\label{Jinnerin}
\Theta(z)J\Theta(z)^* &\leq& J,~~~ |z|>1,\\ 
\label{Jinneron}
\Theta(z)J\Theta(z)^* &=& J,~~~ |z|=1.
\end{eqnarray}
The simplest $J$-lossless functions are the constant $J$ unitary
  matrices  $H$ satisfying $H^*JH=J$.

A $p\times p$ rational matrix function $G(z)$ is called lossless or
conjugate inner  (resp. inner), if and only if
\begin{equation}
\label{Inner}
G(z)G(z)^*\leq I_p,~~~|z|>1~ ({\rm resp.~} |z|<1),
\end{equation}
with equality on the circle. The transfer function of a lossless system 
is  a lossless function. A lossless function can have no pole on the unit
circle  and the identity $G^\sharp(z)G(z)= I_p$ for $|z|=1$,  
extends by analytic continuation 
 to all points where both $G(z)$ and  $G^\sharp(z)$ are  analytic.
 Therefore, the function  $G(z)^{-1}$ agrees with $G^\sharp(z)$ and
 is inner.

We denote by $\cL_n^p$ the set of 
$p\times p$ lossless functions of McMillan degree $n$,
by $\cR\cL_n^p$ the subset of real functions 
and by $\UU(p)$ the
set of $p\times p$ constant unitary matrices. The McMillan degree will be
denoted by $\deg$.

An important property of a lossless function is that if 
\[G(z)=C(zI_n-A)^{-1}B+D,\]
is a balanced realization (it always exists, see \cite{Getal}), then the associated {\em
  realization matrix} 
\begin{equation}
\label{realizationmatrix}
R=\left[ \begin{array}{cc}  D & C \\  B & A \end{array}
  \right]
\end{equation}
is {\em unitary}. Lossless
functions can thus be represented by unitary realization matrices.  Conversely,
if the realization matrix associated with a realization of order $n$ of some
$p\times p$ rational function $G(z)$ is unitary, then $G(z)$ is lossless of
McMillan degree less or equal to n. For these questions, we refer the reader to
\cite{PHO} and the bibliography therein.

Along with a  $2p \times 2p$ rational  function
$\Theta(z)$  block-partitioned  as follows 
\begin{equation}
\label{thetablock} 
\Theta(z) = \left[ \begin{array}{cc} \Theta_{11}(z) & \Theta_{12}(z) \\
\Theta_{21}(z) & \Theta_{22}(z) \end{array} \right],
\end{equation}
with each
block of size $p \times p$, we associate the
linear fractional transformation   ${T}_{\Theta}$
which acts  on $p \times p$ rational functions
$F(z)$ as follows:
\begin{equation}
\label{LFT}
  {T}_{\Theta}(F) =
  [\Theta_{11}\,F + \Theta_{12}] [\Theta_{21}\,F + \Theta_{22}]^{-1}.
\end{equation}
For a composition of linear fractional transformations, it holds that
$T_\Theta\circ T_\Psi= T_{\Theta\Psi}$.
Linear fractional transformations occur extensively in representation
formulas for the solution of various interpolation problems (\cite{BGR}). To
adapt the results available in the literature for 
functions analytic in the disk to the case of functions analytic outside, we 
use the  relation
\[Q=T_\Theta(R) \Leftrightarrow Q^\sharp=T_{J_1\Theta J_1}(R^\sharp),
J_1=\left[\begin{array}{cc}0 & I_p\\
I_p & 0 \end{array}\right].\]
In particular, we have the following result, 
stated for inner functions for example in \cite[Lemma 3]{F-O98}:
\begin{theorem}
If $\Theta(z)$ is a  $J$-lossless  matrix function, then the map $T_{\Theta}$
sends every lossless function to a lossless function.
\end{theorem}

\section{Nudelman interpolation  for lossless functions}
The  Nudelman interpolation problem  is to find a $p\times p$  rational
 lossless function $G(z)$ which satisfies an interpolation 
condition of the form 
\begin{equation}
\label{Nudelman}
\frac{1}{2i\pi}\int_{\bf T} G^\sharp(z) U 
\left(z\,I_\delta- W\right)^{-1} dz
=V,
\end{equation} 
where $(U,W)$ is an  observable pair and $W$ is stable ($U$ is
$p\times \delta$ and $W$ is $\delta\times \delta$).  Note that if $W$ is a diagonal matrix, this
problem reduces  to  a Nevanlinna-Pick problem.
 
It is well-known that there exists a rational lossless function $G(z)$ satisfying 
the interpolation condition (\ref{Nudelman}) if and only if the solution 
$P$ of the symmetric Stein equation 
\begin{equation}
\label{Stein}
P-W^*PW=U^*U-V^*V
\end{equation}
is  positive definite \cite[Th.18.5.2]{BGR}.
A triple $(W,U,V)$  such that the solution $P$ of (\ref{Stein}) is  positive
definite,  will be called an {\em admissible Nudelman data set}.
 A  $2p\times 2p$ $J$-lossless function  can then  be 
 built from  $(W,U,V)$:
\begin{equation}
\label{Theta}
\begin{array}{l}
\Theta_{W,U,V}(z)=\\
\left[I_{2p}-(z-1)C(z\,I_\delta- W)^{-1}{P}^{-1}(I_\delta-
W)^{-*}C^*J\right]
\end{array}
\end{equation} 
where $C=\left[\begin{array}{cc}U\\V\end{array}\right]$.

\begin{theorem}
\label{SolNudelman}
Let $(W,U,V)$ be  an admissible Nudelman data set, and let $\Theta=\Theta_{W,U,V} H$,
where $\Theta_{W,U,V}$ is given by (\ref{Theta}) and $H$ is an arbitrary
constant $J$-unitary matrix. For every lossless function
$F(z)$, the lossless function 
\begin{equation}\label{SolEq}
G=T_{\Theta}(F)
\end{equation}
satisfies (\ref{Nudelman}) and $\deg G=\deg F + \delta.$
 Conversely, the set of all  lossless solutions $G(z)$ of
(\ref{Nudelman})  is given by  (\ref{SolEq}) 
where   $F(z)$ is an arbitrary lossless function. 
\end{theorem}

\begin{proof}
This result is a particular case of   \cite[Th.18.5.2]{BGR}, 
which  describes  all the  Schur functions which are solutions of a Nudelman
interpolation problem.  
\end{proof}

Let $\Lambda$ and $\Pi$ be $p\times p$ unitary matrices. Then the following
important relations are satisfied
\begin{equation}
\label{thetaprop}
\left[
    \begin{array}{cc}
\Lambda & 0 \\ 0 & \Pi 
    \end{array}\right]
\Theta_{W,U,V} 
\left[
    \begin{array}{cc}
\Lambda^* & 0 \\ 0 & \Pi^* 
    \end{array}\right]
=\Theta_{W,\Lambda U,\Pi V}
\end{equation}
\begin{equation}
\label{LFTprop}
 T_{\Theta_{W,\Lambda U,\Pi V}}(\Lambda F(z) \Pi^*)= \Lambda
 T_{\Theta_{W,U,V}}(F(z))\Pi^*
\end{equation}

\section{Balanced realizations}
\label{BalancedRealizations}
The aim of this section is to choose the arbitrary $J$-unitary factor $H$ in
Theorem \ref{SolNudelman},
so that the linear fractional transformation $\widetilde
G=T_{\Theta_{W,U,V}H}(G)$ yields a simple and powerful
construction for  balanced realizations as in (\cite{PHO}).

Let  $\U$ and $\V$ be  $(p+\delta)\times(p+\delta)$ unitary matrices partitioned as
follows: 
\begin{equation}
\label{UVblocks}
\U~=~\left[\begin{array}{cc} \MU &  \alphaU \\
                               \betaU^* & \kU
\end{array}\right],~~~\V~=~\left[\begin{array}{cc} \MV & \alphaV \\
                               \betaV^* & \kV
\end{array}\right],
\end{equation}
where  $\kU$ and $\kV$ are $\delta\times \delta$, $\alphaU$, $\alphaV $, $\betaU$ and
$\betaV$ are $p\times \delta$ and  $\MU$ and $\MV$ are $p\times p$, and put 
\begin{equation}
\label{MAlphaBeta}
M=\left[\begin{array}{cc}\MU &
    0\\0&\MV\end{array}\right],~~{\bf\alpha}=\left[\begin{array}{c} \alphaU\\
    \alphaV \end{array}\right],~~\beta=\left[\begin{array}{c} \betaU\\
    \betaV \end{array}\right].
\end{equation}

\begin{proposition} Let $\U$ and $\V$ be unitary matrices block-partitioned as 
  in (\ref{UVblocks}). Assume that $\kV\,z-\kU$ is
invertible. 
Given a $p\times p$ proper rational transfer function $G(z)$ and a minimal realization
$G(z)=D+C(zI_k-A)^{-1}B$, the formula
\begin{equation}
\label{statespacerecursion2}
\begin{block}{cc}\widetilde D & \widetilde C\\
 \widetilde B & \widetilde A\end{block}
 = \left[ \begin{array}{cc} \U & 0 \\ 0 & I_{k} \end{array} \right]
  \left[ \begin{array}{ccc} D & 0 & C\\ 0 & I_\delta & 0  \\ B &0 & A\end{array}
  \right] \left[ \begin{array}{cc} \V^* & 0 \\ 0 & I_{k} \end{array}
  \right],
\end{equation}
in which $D$, $\widetilde D$ are $p\times p$,  $A$ 
is $k\times k$, $\widetilde A$ is  $(\delta+k)\times (\delta+k)$,
defines a mapping
\[G(z)\to \widetilde
G(z)=\widetilde D+\widetilde C(zI_{\delta+k}-\widetilde A)^{-1}\widetilde B. \]
This mapping coincides with the linear fractional transformation
$\widetilde G
={T}_{\Phi_{\U,\V}}(G),$
associated with the $2p\times 2p$  $J$-lossless 
function 
\begin{equation}
\label{PhiUV}
\Phi_{\U,\V}(z)=M+
\alpha(\kV\,z-\kU)^{-1}\beta^*J\left[\begin{array}{cc}I_p&0\\0&z\,I_p
\end{array}\right].
\end{equation} 
\end{proposition}
\begin{proof}
The case $\delta=1$ has been studied in \cite{PHO}. It is easily verified that
(\ref{statespacerecursion2}) defines a mapping since  
the function $\widetilde G(z)$ does not depend on the choice of the minimal
realization $(A,B,C,D)$ of $G(z)$.\\
We shall use a well-known formula for  the inverse of 
a block matrix \cite[sec.0.2]{Dym}. Assuming that the block  $d$ is invertible, 
the inverse of a block matrix is given by
\begin{eqnarray}
\label{InverseMatrix}
\lefteqn{
\begin{block}{cc}a & b\\c &
  d \end{block}^{-1}} \\&&=
\begin{block}{cc}(a^\times)^{-1} & -(a^\times)^{-1}b d^{-1}\\
- d^{-1} c (a^\times)^{-1} & [d^{-1}+d^{-1}c(a^\times)^{-1}bd^{-1}]
\end{block}\nonumber
\end{eqnarray}
where $a^\times=a-bd^{-1}c$ is known as the Schur complement of $a$.
In particular, if 
\[\Gamma(z)= \begin{block}{cc}\widetilde D & \widetilde C\\
 \widetilde B & (\widetilde A-z I_{\delta+k})\end{block}\]
then
$\widetilde G(z)^{-1}=\begin{block}{cc}I_p & 0\end{block}
\Gamma(z)^{-1}\begin{block}{c}I_p 
  \\ 0\end{block}. $
By (\ref{statespacerecursion2})
\begin{eqnarray*}
\Gamma(z)
&=&\left[ \begin{array}{cc} \U & 0 \\ 0 & I_{k} \end{array}
\right] 
 \left[ \begin{array}{ccc} D & 0 & C\\ 0 & I_\delta & 0  \\ B &0 &
    A\end{array} \right]
  \left[ \begin{array}{cc} \V^* & 0 \\ 0 & I_{k} \end{array}
  \right]-\begin{block}{cc}0 & 0\\
0 & z I_{\delta+k}\end{block}\\
&=&\Gamma_0(z)\left[ \begin{array}{cc} \V^* & 0 \\ 0 & I_{k}
\end{array} 
  \right],
\end{eqnarray*}
\[\Gamma_0(z)=
 \begin{block}{ccc}
M_u D & \alpha_u & M_u C\\
\beta_u^* D - z \beta_v^* & \kU - z \kV & \beta_u^*C\\
B & 0 & A-z I_k \end{block}.\]

The block matrix 
\[d=\begin{block}{cc}
\kU - z \kV & \beta_u^*C\\
 0 & A-z I_k \end{block}\]
is invertible and the Schur complement of $M_u D$ can be computed as
\[\begin{array}{l}
M_uD-\begin{block}{cc}\alpha_u & M_uC\end{block}d^{-1}\begin{block}{c}
\beta_u^*D-z\beta_v^*\\B \end{block}\\
=\Phi_{11}(z) G(z) + \Phi_{12}(z),
\end{array}\]
where $\Phi_{11}(z)=M_u-\alpha_u(\kU-\kV\,z)^{-1}\beta_u^*$ and
$\Phi_{12}(z)=+\alpha_u(\kU -\kV\,z)^{-1}\beta_v^* z $ are precisely the blocks of
the function defined by (\ref{PhiUV}).
 Still using  (\ref{InverseMatrix}) to compute $\Gamma_0(z)^{-1}$,
we get
\[\widetilde G(z)^{-1}=\begin{block}{ccc} M_v & \alpha_v & 0\end{block}
\begin{block}{c}(\Phi_{11}\,G+\Phi_{12})^{-1}\\
-d^{-1}c (\Phi_{11}\,G+\Phi_{12})^{-1}\end{block},\]
which gives
\[\widetilde G(z)^{-1}=(\Phi_{21}\,G+\Phi_{22})(\Phi_{11}\,G+\Phi_{12})^{-1},\]
or equivalently $\widetilde G=T_{\Phi_{\U,\V}}(G)$.
It can be easily established that
\begin{eqnarray*}
\lefteqn{J-\Phi_{\U,\V}(z)J\Phi_{\U,\V}(\lambda)^*}\\
&=&(1-\bar\lambda\,z)\alpha(\kV\,z-\kU)^{-1}
(\kV\,\lambda -\kU)^{-*}\alpha^*,
\end{eqnarray*}
and thus $\Phi_{\U,\V}$ is $J$-lossless.
\end{proof}

{\bf Remark.} A state-space formula of the form (\ref{statespacerecursion2}) associated
with some  linear fractional transformation has been used in \cite{Horiguchi}
to describe all the positive real functions which interpolate given
input-output characteristics. 

\begin{proposition}
Let $(W,U,V)$ be some admissible Nudelman data set. There exist unitary
$(p+\delta)\times(p+\delta)$   matrices  $\U$ and $\V$ and a $2p\times 2p$ constant
$J$-unitary matrix $H_{\U,\V}$  such that
\begin{equation}
\label{ThetaPhi} 
\Theta_{W,U,V}\,H_{\U,\V}=\Phi_{ \U, \V}
\end{equation}
\end{proposition}

\begin{proof}
If (\ref{ThetaPhi}) is satisfied,  since $\Theta_{W,U,V}(1)=I_{2p}$, 
the matrix $H_{\U,\V}$ must be given by 
 \begin{equation}
\label{HWUV}
H_{\U,\V}=M+
{\bf\alpha}
(\kV-\kU)^{-1}\beta^*J.
\end{equation}
Moreover, the function $\Phi_{\U,\V}$ cannot have a
pole on the circle and  can be rewritten 
\[
\begin{array}{l}
\Phi_{\U,\V}(z)H_{\U,\V}^{-1}=\\
\left[I_{2p}- (z-1)\,
{\bf\alpha}
(\kV\,z-\kU)^{-1}
( \kV- \kU)^{-*}{\bf\alpha}^*J\right].
\end{array}\]
The representation (\ref{Theta}) of a $J$-lossless function being  unique up
to a similarity transformation (\cite{BGR}), there must exist a transformation $T$ such that 
$P= T^*T$, $\kU \kV^{-1}=T W T^{-1}$, $\alpha \kV^{-1}= C T^{-1}$. 
 The matrix $T$ is thus a  square root of $P$.   
Since the  matrix $\V$  must be unitary, we must have
 $\alphaV^*\alphaV+\kV^*\kV=I_\delta$, or else
\begin{equation}
\label{kveq}
(T^{-*}V^*V T^{-1}+I_\delta)^{-1}=\kV\kV^*.
\end{equation}
The matrix  $T^{-*}V^*V T^{-1}+I_\delta$ being
positive definite, this equation has  solutions  and $\kV$ being one of
these, 
we can set
\begin{equation}
\label{CUCV}
\left\{
\begin{array}{rcl}
\alphaU&=& \tilde U \kV\\
\kU&=& \tilde W \kV\\
\alphaV&=&  \tilde V \kV
\end{array}\right.,
\end{equation}
in which
\begin{equation}
\label{UVWtilde}
\left\{\begin{array}{rcl}
\tilde U&=&U T^{-1} \\
\tilde W&=&T W T^{-1}\\
\tilde V&=&V T^{-1}
\end{array}\right. .
\end{equation}
These definitions imply that $\alphaU^*\alphaU+\kU^*\kU=I_\delta$
as required and  the columns 
 $\left[\begin{array}{c}\alphaU\\  
    \kU\end{array}\right]$ and $\left[\begin{array}{c}\alphaV\\  
    \kV\end{array}\right]$ can be completed into unitary matrices $\U$ and
$\V$. The columns $\left[\begin{array}{c}\MU\\  
    \betaU^*\end{array}\right]$ and $\left[\begin{array}{c}\MV\\  
    \betaV^*\end{array}\right]$ can be  determined up to some right $p\times p$
unitary matrices.
\end{proof}

\section{Explicit formulas for $\U$ and $\V$.}
\label{ExplicitFormulas}
An observable pair
$(U,W)$ such that $W$ is stable is  called output
normal if it satisfies 
\begin{equation}
\label{OutputNormal}
U^*U+W^*W=I_\delta.
\end{equation}
Note that two equivalent triples $(W,U,V)$ and $(TWT^{-1},UT^{-1},VT^{-1})$
give the same interpolation condition (\ref{Nudelman}), so that we can assume
that the pair $(U,W)$  in (\ref{Nudelman}) is output normal. From now on,
this normalization condition will be imposed to the admissible Nudelman data sets. 
 Explicit formulas for $\U$ and $\V$ can then be given  which ensure the smoothness 
of our parametrization. They  have been used for
implementation (see section \ref{atlasPaulo}).

It has been proved that the second block columns of $\U$ and $\V$ 
are given by (\ref{CUCV}) and
(\ref{UVWtilde}) in which $T$ is a square root of $P$ (the solution to
(\ref{Stein})) and $\kV$ a solution to (\ref{kveq}). We {\em choose} 
the {\em uniquely determined Hermitian positive square roots}
$T=P^{1/2}$ and 
$\kV=(I_\delta+\tilde V^*\tilde V)^{-1/2}$. 
It remains to specify the completion  of these block columns into unitary
matrices.  
The matrix $\MV$ satisfies
\[\MV\MV^*= I_p-\tilde V (I_\delta+\tilde V^*\tilde V)^{-1}\tilde V^*,\]
and it is easily seen that 
\[I_p-\tilde V(I_\delta+\tilde V^*\tilde V)^{-1}\tilde V^*=(I_p+\tilde V\tilde
V^*)^{-1}\] 
is positive definite. Thus we can choose
\[\MV =(I_p+\tilde V\tilde V^*)^{-1/2},\]
 and thus $\beta_V^*= -\tilde V^*(I_p+\tilde V\tilde V^*)^{-1/2}$,
so that we can set
\begin{equation}
\label{matV}
\V=\left[\begin{array}{cc}
(I_p+\tilde V\tilde V^*)^{-1/2}
& \tilde V(I_\delta+\tilde V^*\tilde V)^{-1/2} \\
-\tilde V^*(I_p+\tilde V\tilde V^*)^{-1/2}
& (I_\delta+\tilde V^*\tilde V)^{-1/2} 
\end{array}\right].
\end{equation} 
 The construction of a matrix $\U$ is more involved since the matrix
$I_p-\alphaU\alphaU^*$ may fail to be positive definite. This is the 
case for example when $W$ is the zero matrix, then $\alphaU=U$ and 
$I_p-UU^*$ is not invertible.  
However, when $V$ is zero, since $W$ is stable, $I_\delta-W^*$ is invertible and
there is a simple way to construct a unitary matrix 
\begin{equation}
\label{U0}
 \U_0=\left[\begin{array}{cr}
X & U\\
Y & W
\end{array}
\right], 
\end{equation}
\begin{equation}
\label{XY}
\left\{\begin{array}{rcl} X&=& I_p-U(I_\delta-W^*)^{-1}U^*,\\
 Y&=&(I_\delta-W)(I_\delta-W^*)^{-1}U^*.\end{array}\right.
\end{equation}
Consider the  matrix 
\[
\begin{block}{cc} 
I_p & 0\\
0 & T
\end{block}
\U_0
\begin{block}{cc} 
I_p & 0\\
0 & T^{-1}
\end{block}= \begin{block}{cc}
X & \tilde U\\
T Y & \tilde W
\end{block}.
\]
The problem is now to find  a right factor of the form 
$\begin{block}{cc} \star & 0\\ \star & \kV\end{block}$
which makes it into a unitary matrix.
Let
\[\begin{block}{cc}
N & L^*\\
L & K 
\end{block}=\begin{block}{cc}
X & \tilde U\\
T Y & \tilde W
\end{block}^*\begin{block}{cc}
X & \tilde U\\
T Y & \tilde W
\end{block}.\]
 A classical method consists of  writing a
Cholesky factorization using the following well-known factorization of a
$(p+\delta)\times (p+\delta)$ block matrix 
\begin{eqnarray*}
\begin{block}{cc}
N & L^*\\
L & K 
\end{block}
&=&\begin{block}{cc}
I_p & L^*K^{-1}\\
0 & I_\delta
\end{block}
\begin{block}{cc}
Z^{-1} & 0\\
0 & K
\end{block}
\begin{block}{cc}
I_p & 0\\
 K^{-1} L & I_\delta
\end{block}\\
&=&
\begin{block}{cc}
Z^{-1/2} & 0\\
 K^{-1/2} L & K^{1/2}
\end{block}^*
\begin{block}{cc}
Z^{-1/2} & 0\\
 K^{-1/2} L & K^{1/2}
\end{block}
\end{eqnarray*}
where $Z=(N-L^*K^{-1}L)^{-1}$. 
By (\ref{InverseMatrix}), $Z$ is the left upper block of 
\[\begin{block}{cc}
N & L^*\\
L & K 
\end{block}^{-1}=
\begin{block}{cc}
I_p & 0\\
0 & T
\end{block}
\U_0^*
\begin{block}{cc}
I_p & 0\\
0 & P^{-1}
\end{block}
\U_0
\begin{block}{cc}
I_p & 0\\
0 & T^*
\end{block}\]
and can be computed as
\begin{equation}
\label{Z}
Z=X^*X+Y^*P^{-1}Y.
\end{equation}
The matrices $L$ and $K$ are given by
\begin{eqnarray}
\label{L}
L&=& \tU^*X+\tW^*T Y\\
\label{K}
K&=& \tU^*\tU+\tW^*\tW=I_\delta+\tV^*\tV.
\end{eqnarray}
 Note that the matrices $Z$ and $K$ are actually positive
definite and that $K^{-1/2}=\kV$ as desired. Thus, we can set
\begin{equation}
\label{matU}
\U=\begin{block}{cc}
X & \tilde U\\
T Y & \tilde W 
\end{block}
\begin{block}{cc}
Z^{1/2} & 0 \\
-K^{-1}L Z^{1/2} & K^{-1/2}
\end{block}
\end{equation} 

This proves the following result:
\begin{proposition}
\label{tau}
Let $(W,U,V)$ be some admissible Nudelman data set satisfying
(\ref{OutputNormal}).  
Define the map $\tau:~(W,U,V)\to(\U,\V),$
where $\U$ and $\V$ are the unitary matrices 
given by (\ref{matU}) and (\ref{matV}), where $\widetilde U$, $\widetilde
V$, $\widetilde W$ are given by (\ref{UVWtilde}),  $K, L, Z$  by 
(\ref{K}),(\ref{L}),(\ref{Z}) and $X, Y$ by (\ref{XY}), in all of which
 $T$ is the positive square root of $P$, the solution to 
 (\ref{Stein}).\\ 
Then, the $J$-lossless function
\begin{equation}
\label{ThetaHat}
\hat\Theta_{W,U,V}=\Theta_{W,U,V} H_{\U,\V},
\end{equation}
$H_{\U,\V}$ being given by (\ref{HWUV}), coincides with $\Phi_{\U,\V}$. 
\end{proposition}

 Let $\Lambda$, $\Pi$ be $p\times p$ unitary matrices  
and $\Sigma$ be a $\delta\times \delta$ unitary matrix.
Noting that 
$(\Lambda Z \Lambda^*)^{1/2}=\Lambda Z ^{1/2}\Lambda^*$,
it is easily verified that  
\begin{eqnarray*} 
\lefteqn{\tau(W,\Lambda U,\Pi V)}\\
&&= \left(\begin{block}{cc}
\Lambda & 0\\
0 & I_\delta
\end{block}\U\begin{block}{cc}
\Lambda^* & 0\\
0 & I_\delta
\end{block},\begin{block}{cc}
\Pi & 0\\
0 & I_\delta
\end{block}\V\begin{block}{cc}
\Pi^* & 0\\
0 & I_\delta
\end{block}\right),
\end{eqnarray*}
so that the  $J$-lossless function $\hat\Theta_{W,U,V}$ also satisfies
(\ref{thetaprop}) and (\ref{LFTprop}).

We also have that
\begin{eqnarray*} 
\lefteqn{\tau(\Sigma^*W\Sigma,U\Sigma,V\Sigma)}\\
&&= \left(\begin{block}{cc}
I_p & 0\\
0 & \Sigma^*
\end{block}\U\begin{block}{cc}
I_p & 0\\
0 & \Sigma
\end{block},\begin{block}{cc}
I_p & 0\\
0 & \Sigma^*
\end{block}\V\begin{block}{cc}
I_p & 0\\
0 & \Sigma
\end{block}\right),
\end{eqnarray*}
so that 
\begin{equation}
\label{similarity}
\hat\Theta_{\Sigma^*W\Sigma,U\Sigma,V\Sigma}=\hat\Theta_{W,U,V}.
\end{equation}

\begin{corollary}
A unitary matrix realization $\widetilde R$ of $\widetilde G
={T}_{\hat\Theta_{W,U,V}}(G)$ can be computed from a unitary matrix
realization  $R$ of $G(z)$ by (\ref{statespacerecursion2}) in which $\U$ and
$\V$ are the unitary matrices
$(\U,\V)=\tau(W,U,V)$ defined  in Proposition \ref{tau}.
\end{corollary}

{\bf Remark.}  Instead of $\U$ and $\V$, we may have chosen
\begin{eqnarray*}
 \hat\U&=& \begin{block}{cc}
I_p & 0 \\
0 & O_1
\end{block}\U
\begin{block}{cc}
H_1 & 0 \\
0 & O_2
\end{block}\\
\hat\V&=& \begin{block}{cc}
I_p & 0 \\
0 & O_1
\end{block}\V
\begin{block}{cc}
 H_2 & 0 \\
0 &O_2
\end{block},
\end{eqnarray*}
in which $O_1$, $O_2$, $H_1$ and $H_2$ are  unitary matrices:\\
(1) the matrix $O_1$ corresponds to the choice of  $O_1P^{1/2}$ instead of
$P^{1/2}$.  This choice leave the linear fractional transformation unchanged
($\Phi_{\hat\U,\hat\V}=\Phi_{\U,\V}$), 
but  changes the realization 
  $\widetilde R=(\widetilde A,\widetilde B,\widetilde C,\widetilde D)$ into
$(O_1\widetilde AO_1^*,O_1\widetilde B,\widetilde CO_1^*,\widetilde D)$,  a
  similar one.\\
(2) the choice of $\kV=(\tV^*\tV+I_p)^{-1/2}O_2$ instead
of  $(\tV^*\tV+I_p)^{-1/2}$  has no effect.\\
(3) The unitary matrices $H_1$ and $H_2$ correspond to another completion of the first columns of
$\V$ and $\U$. This choice  
changes the linear fractional transformation since
\[\Phi_{\hat\U,\hat\V}=\Phi_{\U,\V}\begin{block}{cc}
 H_1 & 0 \\
0 &H_2
\end{block}\]
and thus produces a non-similar realization.

{\bf Remark.}
Note that if $V=0$, since the pair $(U,W)$ satisfies 
$U^*U+W^*W=I_\delta$,
we have that $P=I_\delta$ and   $T=I_\delta$ too. Thus, $\U=\U_0$ defined by
(\ref{U0}) and $\V=I_{p+\delta}$,
so that  the recursion (\ref{statespacerecursion2}) becomes
\begin{equation}
\label{v=0}
\left[\begin{array}{c|c}
\widetilde D & \widetilde C\\
\hline 
\widetilde B & \widetilde A
\end{array}\right]=
\left[\begin{array}{c|cc}
\MU D & U  & \MU C\\
\hline
\betaU^* D &  W &\betaU^* C \\
B& 0 & A
\end{array}\right].
\end{equation}

\section{Charts from a Schur algorithm.}

Recall that a manifold  is a topological space that looks locally like the
"ordinary" Euclidean space ${\mathbb R}^N$: near every point of the space, we have a coordinate system or
chart.  The number $N$ is the dimension of the manifold. 
It has been proved in \cite[Th.2.2]{ABG} that $\cL_n^p$  is a
smooth manifold of dimension $p^2+2np$ 
embedded in the Hardy space $H_q^{p\times p}$, for $1\leq q \leq \infty$  ($\cR\cL_n^p$
is a smooth manifold of dimension  $\frac{p(p-1)}{2}+np$). The topology  on
$\cL_n^p$ is that induced by the $L^q$ norm on  $H_q^{p\times
  p}$.  In \cite{ABG} atlases of charts have been constructed from a Schur
algorithm associated with Nevanlinna-Pick interpolation. We generalize this
construction to the case of Nudelman interpolation.

Let $\sigma=\left((U_1,W_1),~(U_2,W_2),\ldots,(U_l,W_l)\right)$
 be a sequence  of  output normal  pairs (see section \ref{ExplicitFormulas}),
$W_j$ is $n_j\times n_j$,  $U_j$ is $p\times n_j$, and
\[\sum_{j=1}^{l}n_j=n.\]
From a given lossless function $G(z)$ of degree $n$, 
 a sequence  of lossless functions of decreasing degree $G_l(z)=G(z),G_{l-1}(z),\ldots
$ can be constructed following a Schur algorithm:  assume that
$G_j(z)$ has been constructed  and put   
\[V_j=\frac{1}{2i\pi}\int_{\bf T} G_j^\sharp(z) U_j 
\left(z\,I_{n_j}- W_j\right)^{-1} dz. \]
If the solution $P_j$ to the symmetric Stein equation 
\[P_j-W_j^*P_jW_j=U_j^*U_j-V_j^*V_j\]
is positive definite, then from Theorem \ref{SolNudelman}, a 
lossless function $G_{j-1}(z)$  is defined by
\[G_j=T_{\hat\Theta_{W_j,U_j,V_j}}(G_{j-1}).\]
If $P_j$ is not positive definite, the construction stops.

A chart $(\cD,\phi)$ of $\cL_n^p$ is
attached with a sequence $\sigma$ of output normal pairs and with a chart
$(\cW,\psi)$ of $\UU(p)$ as follows: \\
A function $G(z)\in\cL_n^p$ belongs to the domain $\cD$ of the chart if  the
Schur algorithm allows to construct a complete sequence of 
lossless functions,
 \[G(z)=G_l(z), G_{l-1}(z)\ldots,G_0,\]
where  $G_0$ is a constant lossless matrix in $\cW\subset \UU(p)$.\\
The local coordinate map $\phi$ is defined by 
\[\phi: G(z)\in\cD \to \left(V_1,V_2,\ldots,V_l,\psi(G_0)\right),\]
and the  interpolation matrices $V_j$ are called the {\em Schur parameters} of the
function in the chart.
 
\begin{theorem}
A family of charts $(\cD,\phi)$ defines  an atlas of $\cL_n^p$ provided 
the union of their domains covers $\cL_n^p$.  
\end{theorem}

\begin{proof}
The proof is analogous to the proof of Th.3.5 in \cite{ABG}. The domain $\cD$ of a
chart is open and  the map $\phi$ is a diffeomorphism. This relies on the
fact that $\hat\Theta_{W,U,V}$ depends smoothly on the entries of $V$.
\end{proof}

Atlases for the quotient $\cL_n^p/\UU(p)$ are obtained using the properties 
(\ref{thetaprop}) and (\ref{LFTprop}). If $G(z)$ has Schur parameters
$\left(V_1,V_2,\ldots,V_l\right)$ and constant unitary matrix $G_0$ in a
given chart, and  if $\Pi\in \UU(p) $, then $G(z)\Pi^*$ has Schur parameters 
$\left(\Pi V_1,\Pi V_2,\ldots,\Pi V_l\right)$ and constant unitary matrix
$G_0\Pi^*$ in the same chart. The quotient can be performed within a chart by
imposing the last  
constant lossless matrix $G_0$ in the Schur algorithm to be the identity
matrix.  

In a chart, a  balanced realization $R$ of $G(z)=\phi^{-1}(V_1,V_2,\ldots,  
\psi(G_0))$ can be  computed  from the parameters using 
 the Schur sequence: let $R_0=G_0$, a realization $R_j$  of $G_j(z)$ is 
obtained from a realization $R_{j-1}$ of $G_{j-1}(z)$ by
formula (\ref{statespacerecursion2})  in which
$(\U,\V)=(\U_j,\V_j)=\tau(W_j,U_j,V_j)$ (see Proposition \ref{tau}). This
process allows to select for each $G(z)$ in the domain of the chart 
a unique balanced realization within the
equivalence class, and then the map
\[R\to (V_1,V_2,\ldots,
\psi(G_0))\]
is a canonical form. However, the domain of this canonical form is not
easily characterized and it is in general difficult to decide if a given
realization is in canonical form with respect to a chart. This can be done in 
some particular situations that will be studied in the
following section.

\section{Some particular atlases.}
We describe three atlases which  all present some 
interest from the optimization viewpoint. The first one is for complex
functions and it involves only Schur steps in which the degree is increased by
one. It allows for a search strategy of local minima by induction on the
degree, which can be very helpful in some difficult optimization
problems. The second one is the analog for real-valued functions. The third
one involves only one Schur step and provides very simple and natural 
canonical forms.

A chart in which all the Schur
parameters $V_j$ are zero  matrices for $G(z)$ is called an  {\em adapted
  chart} for $G(z)$. Such a chart presents a great interest from
an optimization viewpoint. The optimization process  starts, in an adapted
chart, at the origin  and thus far from the boundary where a change
of chart is necessary. For each atlas,  a simple method  to find an
adapted chart is given.

\subsection{An atlas for $\cL_n^p$ (complex lossless functions)}
\label{atlasPHO}
Consider the charts associated with  sequences of output normal pairs 
$(u_1,w_1),~(u_2,w_2),\ldots,(u_n,w_n)$ in which the $w_j$'s are complex
numbers.
In this case, the Nudelman
interpolation condition (\ref{Nudelman}) can be rewritten as a
Nevanlinna-Pick interpolation  condition 
\[ G(1/{\bar w_j})^*u_j=v_j.\]
This is the atlas described in \cite{PHO}. However, the normalization
conditions differ. In \cite{PHO} the $p$-vectors
$u_j$ have norm one, while in this work, the pairs $(u_j,w_j)$ are output
normal  
(\ref{OutputNormal}).  

{\bf Remark.} Note that,  when $n_j>1$  the normalization condition  $U_j$
unitary  (a possible  generalization of  $\|u_j\|=1$) 
cannot be chosen since the matrix  $U_j^*U_j$ can be singular.  

In view of (\ref{v=0}), an adapted chart can be computed from a realization
in Schur form. 

\begin{lemma} Let $\widetilde G(z)\in \cL_n^p$ and 
 let $\widetilde R=(\widetilde A,\widetilde B,\widetilde C,\widetilde D)$ be  a balanced
realization  of  $G(z)$ in Schur form ($\widetilde A$ upper triangular). Let 
\[\widetilde A=\left[\begin{array}{cl} w & a^*\\
                                0    & A
\end{array}\right],~~~ 
{\widetilde B}=\left[\begin{array}{l} b^* \\ B
\end{array}\right],~~~{\widetilde C}=\left[\begin{array}{cc} u & \hat C                                                    
\end{array}\right],\]
where $w\in\CC$, $u,~b\in\CC^p$,   and $a\in\CC^{n-1}$. 
Then, $\widetilde G=T_{\hat\Theta_{w,u,0}}(G)$ for some 
lossless function $G(z)$. A realization $R$ of $G(z)$ 
  can be computed by reverting
(\ref{statespacerecursion2}). It is still in Schur form 
and given by   $R=(A,B,C,D)$,  where
\begin{eqnarray*}
C&=&\hat C+(1-w)^{-1} u a^*\\
D&=&\widetilde D+(1-w)^{-1} u b^*.
\end{eqnarray*}
\end{lemma}

This process can  be repeated. It  provides  a sequence
of output normal  pairs $(u_j,w_j)$, the $w_j$'s being the  eigenvalues of
$\widetilde A$. In the corresponding  chart   the Schur
parameters of $G(z)$ are the zero $p$-vectors $v_n=\ldots=v_1=0$.

\subsection{An atlas for $\cR\cL_n^p$ (real lossless functions)}
To deal with real functions we  consider the charts associated with
sequences of output normal pairs  $(U_1,W_1)$, $(U_2,W_2),\ldots, (U_n,W_n)$
 in which the $W_j$'s are 
either real numbers or real $2\times 2$ matrices with complex conjugate
eigenvalues and the $U_j$'s are real matrices. The parameters (the matrices
$V_j$) are then restricted to be {\em real}.

As previously,  an adapted chart for a given lossless function
$\widetilde G(z)\in\cR\cL_n^p$  can be obtained  from a realization in real Schur form.

\begin{lemma} Let $\widetilde G(z)\in \cR\cL_n^p$ and 
 let $\widetilde R=(\widetilde A,\widetilde B,\widetilde
C,\widetilde D)$ be a  balanced realization of $G(z)$ in {\em real
  Schur form}
\[\widetilde A=
\left[\begin{array}{cccc}
W_l&\star & \cdots & \star \\
0 & W_{l-1} & \ddots & \vdots\\
\vdots&\ddots &\ddots &\star\\
0&\cdots&0& W_1
\end{array}\right],
\]
where for $j=1,\ldots,l$, $W_j$ is either a real number or a $2\times 2$
block with  
complex conjugate eigenvalues.
Let 
\[\widetilde A=\left[\begin{array}{cl} W_l & \hat A^*\\
                                0    & A
\end{array}\right],~~~ \widetilde C=\left[\begin{array}{cc}U_l & \hat C
  \end{array}\right],~~~{\widetilde B}=\left[\begin{array}{l} \hat B^* \\ B
\end{array}\right],\] 
where $U_l$ and $\hat B$ are $p\times n_l$, $n_l$ being the size of $W_l$. 
Then, $\widetilde G=T_{\hat\Theta_{W_l,U_l,0}}(G)$ for some 
lossless function $G(z)$. A realization $R$ of $G(z)$ 
  can be computed by reverting
(\ref{statespacerecursion2}). It is in Schur form 
and given by   $R=(A,B,C,D)$,  where
\begin{eqnarray*}
C&=&\hat C+(I_{n_l}-W_l)^{-1} U_l \hat A^*\\
D&=&\widetilde D+(I_{n_l}-W_l)^{-1} U_l \hat B^*.
\end{eqnarray*}
\end{lemma}

Repeating this process, we 
get a sequence of output normal pairs $(U_j,W_j)$, the $W_j$'s being the
diagonal blocks of 
$\widetilde A$, that index a chart in which the Schur
parameters (which are now matrices of different sizes) 
of $G(z)$ are all zero matrices.

{\bf Remark.} The Schur algorithm attached with $G(z)$ in such an
adapted chart  yields
a Potapov factorization  for real lossless functions
\[G(z)=B_l(z) B_{l-1}(z)\cdots B_1(z),\]
where $B_j$ is the real-valued lossless  function
\begin{eqnarray*}
\lefteqn{B_j(z)=}\\
&& I_p-(z-1)U_j(z\,I_{n_j}-W_j)^{-1}(I_{n_j}-W_j^*)^{-1}U_j^*.
\end{eqnarray*}

\subsection{Lossless mutual encoding.}
\label{atlasPaulo}
For this atlas, we consider  the charts associated with a single output
normal pair 
$(U,W)$ in which $W$ is $n\times n$ and $U$ is $p\times n$. In this 
case, a  solution to (\ref{Nudelman}) can be  directly characterized in state
space form and 
formula (\ref{statespacerecursion2}) recovered independently from the Schur algorithm.  

\begin{proposition}
Let $G(z)=D +C (z I_n -A)^{-1} B$ be a balanced realization of  $G(z)\in
\cL_n^p$.   Let $Q$ be the unique solution to the Stein equation 
\begin{equation}
\label{Q}
Q - A^* Q W = C^* U.
\end{equation}
Then,  the interpolation value $V$  in (\ref{Nudelman}) and the solution $P$ to
(\ref{Stein})  are given by
\begin{equation}
\label{Vcasen}
V=D^* U + B^* Q W,
\end{equation}
\begin{equation}
\label{PQ}
P=Q^*Q.
\end{equation}
The unitary realization matrix $
R=\left[ \begin{array}{cc}  D & C \\  B & A \end{array}
  \right]$  of $G(z)$ can be computed as 
\[R=\hat\U \begin{block}{cc}
G_0 & 0\\
0 & I_n
\end{block}\hat\V^*,~~~G_0\in\UU(p),\]
 in which  $\hat\U$ and $\hat\V$ are unitary matrices   given as in
 Proposition \ref{tau} 
by (\ref{matU}) and (\ref{matV}), where $\widetilde U$, $\widetilde
V$, $\widetilde W$ are given by (\ref{UVWtilde}),  $K, L, Z$  by 
(\ref{K}),(\ref{L}),(\ref{Z}) and $X, Y$ by (\ref{XY}), but in which {\em the square root 
 $T$ of $P$  is now chosen to be $Q$}.
\end{proposition}
{\bf Proof.}
Since $G^\sharp(z)=D^*+B^*\left(\frac{1}{z}I_n-A^*\right)^{-1}C^*$,  the  contour integral (\ref{Nudelman})
can be computed as 
\begin{eqnarray*}
V&=&\frac{1}{2i\pi}\int_{\bf T} G^\sharp(z)  U 
\left(z\,I_n- W\right)^{-1} dz
\\
V&=&\frac{1}{2i \pi}\int_{\bf T}D^*\,U\left(\sum_{j=0}^\infty (z^{-1}
  W)^{j}\right) \\
&&+B^* z \left(\sum_{j=0}^\infty (zA^*)^{j}\right) C^*\,U 
\left(\sum_{j=0}^\infty z^{-j} W^{j}\right)\frac{dz}{z},\\
&=& D^*U+B^*\left(\sum_{j=0}^\infty (A^*)^jC^*UW^j\right) W.
\end{eqnarray*}
 Since $Q$ is given by the convergent series
$Q=\sum_{j=0}^\infty (A^*)^jC^*UW^j,$ (\ref{Vcasen}) is satisfied. 

Consider a unitary completion of the column $\begin{block}{c}
  U\\W\end{block}$, for example the matrix $\U_0 $ given by 
(\ref{U0}). 
Here, $\U_0$ and the unitary realization matrix $R$ have the same size and formulas (\ref{Q}) and
(\ref{Vcasen}) can be rewritten in a matrix form 
\begin{equation}
\label{Upsilon}
R^*
\begin{block}{cc}
I_p & 0 \\
0 & Q 
\end{block} \U_0=
\begin{block}{cc}
\star & V \\
\star  & Q 
\end{block} = \Upsilon.
\end{equation}

The matrix $R$ being unitary, we have that 
\[\Upsilon^*\Upsilon=\U_0^*\begin{block}{cc}
I_p & 0 \\
0 & Q^*Q
\end{block}
\U_0=\begin{block}{cc}
\star & \star\\
\star & V^*V+Q^*Q
\end{block},\]
so that
$U^*U+W^*Q^*QW=V^*V+Q^*Q,$ 
which proves (\ref{PQ}).

We shall use the computations of section \ref{BalancedRealizations},
with $\delta=n$  and $T=Q$ {\em instead of the Hermitian positive square root}
$P^{1/2}$.
 From (\ref{Upsilon}) we get 
\[
R^*
\begin{block}{cc}
I_p & 0 \\
0 & Q 
\end{block} \U_0 \begin{block}{cc}
I_p & 0 \\
0 & Q^{-1} 
\end{block}=
\begin{block}{cc}
\star & \widetilde V \\
\star  & I_p
\end{block},\]
and if $\hat\U$ is the unitary matrix given by (\ref{matU}),
\[R^*\hat\U=\begin{block}{cc}
\star & \widetilde V K^{-1/2}\\
\star  & K^{-1/2}
\end{block}.\]
This matrix is unitary and its second block column coincides with that of
$\hat\V$ 
given by (\ref{matV}). Thus it must be $\hat\V$ up to a right
unitary 
factor of the form 
$\begin{block}{cc} G_0^* & 0\\ 0 & I_p \end{block}$,
for some unitary matrix $G_0$.  
\hfill$\Box$

{\bf Remark.} Note that $\Phi_{\hat\U,\hat\V}=\Phi_{\U,\V}$ ($\U, \V$ defined
in Proposition
\ref{tau}) and thus $G_0$ is 
the  constant unitary matrix  in the Schur algorithm 
$G=T_{\hat\Theta_{W,U,V}}(G_0)$.

The domain of a chart and of the associated canonical form can then be easily
characterized. 

\begin{proposition}
A lossless function $G(z)$, given by a balanced realization $(A,B,C,D)$
 can be parametrized in the chart defined by the pair $(U,W)$ if and only if
the solution $Q$ to the Stein equation
(\ref{Q}) is positive definite. 
A realization $\tilde R$ is  in canonical form with respect to this chart 
if and only if the solution  $Q$ to (\ref{Q})
is $P^{1/2}$,  $P$ being a solution of (\ref{Stein}).
\end{proposition}

The invertibility of the matrix $Q$ is a good measure of the quality of the
chart, the best choice being $Q=I_n$. This choice provides an adapted chart.

\begin{proposition}
The chart associated with the output normal pair $(C,A)$  in a balanced
realization $(A,B,C,D)$ of $G(z)$ is an adapted chart for $G(z)$.
\end{proposition}

\begin{proof}
In this case,  $Q=I_n$, so that $P=I_n$ and $V=0$.
\end{proof}

{\bf Remark.} Let $(\cD,\phi),~\phi:G(z)\to (V,G_0)$
 be the chart associated with the output normal
pair $(U,W)$. Then the chart associated with 
the pair $(U\Sigma,\Sigma^*W\Sigma)$
has same domain $\cD$ and by (\ref{similarity}) 
coordinate map $\phi^\prime:G(z)\to(V\Sigma,G_0)$.
In an atlas, these two  charts play the same role. 

{\bf Remark.}
Equivalence classes of output normal
pairs are in bijection 
with lossless functions in  $\cL_n^p/\UU(p)$  \cite[Cor.2.1]{ABG}. 
The unitary completion  $\U_0$  of the matrix  $\begin{block}{c}
  U\\W\end{block}$ in (\ref{U0}) defines a lossless function 
\begin{equation}
\label{Omega}
\Omega(z)=X+U(zI_n-W)^{-1}Y\in\cL_n^p.
\end{equation}
The canonical form associated with the pair $(U,W)$ depends on this 
completion and 
  is in fact attached with an element of $\cL_n^p/\UU(p)$. This explains why
  this section was called lossless mutual encoding. 

\section{Application to system identification and control.}

 The first application that we consider is  the identification of
 hyperfrequency filters, made of coupled resonant cavities, that are used in
 telecommunication satellites for channel multiplexing. The problem is to
 recover the transfer function of the filter from frequency data. These data are 
 estimate values of the transfer function at  pure imaginary points 
  obtained from the steady-state outputs of the filter to  harmonic
 inputs. A first stage, far from being trivial, consists in computing a
stable matrix transfer function of high order which agrees with the data. 
It is achieved by the software Hyperion, 
also developed at INRIA (\cite{RTCNES}).
Then a rational $L^2$ approximation stage
is performed by the software RARL2\footnote{The software RARL2 is described 
in \cite{CDC02} and available at the web page \\
{\tt \small http:www-sop.inria.fr/apics/RARL2/rarl2-eng.html.}}, in which the
atlas of section \ref{atlasPHO} 
is used. Transfer functions of these filters are  complex functions since  a
particular transformation has been  used to simplify  the model. 
 In  Figure \ref{bodeCNES},
a 8th order model of a MIMO
$2\times 2$ hyperfrequency filter is shown, obtained from 800 pointwise
data. 
\begin{figure}[!hbt]
\begin{center}
\includegraphics[width=7.5cm]{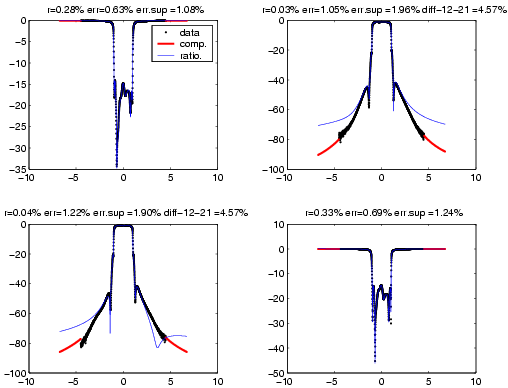}
\end{center}
\caption{CNES $2\times 2$ hyperfrequency filter:
 data and approximant at order 8 (Bode diagram).}
\label{bodeCNES}
\end{figure}
A longstanding cooperation with the space agency CNES resulted in a dedicated
software PRESTO-HF that wraps both HYPERION and RARL2 into a package which is
now fully integrated in the design and tuning process. 

This application and more generally filter design, raises interesting 
new parametrization issues. The physical laws of energy conservation and
reciprocity introduce subclasses of transfer functions which play an
important role in this domain.  These include $J$-inner, Schur (or
contractive), positive real and symmetric functions. In another connection,  
 systems having a particular state space form  must be  handled to
account for some physical properties, like for example the coupling geometry
of a filter. 
We think that Schur analysis could help us to describe
such subclasses and to pave the bridge between the frequency domain (where
specifications are made) and the state-space domain (where the design
parameters live). As a first step in this way,  a  Schur algorithm for
symmetric 
lossless functions, based on  a  two-sided Nudelman interpolation condition,
has 
been presented in \cite{O-CDC}.

 The atlas of section 
\ref{atlasPaulo} has been recently implemented in RARL2, the rational $L^2$
approximation software.  Its effectiveness
has been demonstrated on random 
systems and classical examples from the literature.  A promising  field of
application, in which functions are 
real-valued,  is  multi-objective control.  
 In \cite{Scherer00}, revisited in a chain-scattering
perspective in \cite{O-IFAC}, it is shown that if  the pair
$(C_Q,A_Q)$  of the Youla parameter $Q(z)=D_Q+C_Q(z I-A_Q)^{-1}B_Q$ 
is fixed, then the  search over the parameters $(B_Q,D_Q)$ can be 
 reduced to  an efficiently solvable  LMI problem. Limiting the search of
 the parameter $Q(z)$ to the FIR form
\[Q(z)=Q_0+Q_1\frac{1}{z}+\ldots+Q_p \frac{1}{z^p}\]
provides  solutions to the multi-objective control problem. 
However, this is also the main limitation of the approach as high order
expansions might be necessary, due notably to the fact that the poles
structure is fixed through the pair $(C_Q,A_Q)$. 
Such a drawback  could be avoided if the search was performed over all
the  parameters
$Q(z)$ of  fixed McMillan degree. This could be done    using the
atlas  of  section \ref{atlasPaulo} to parametrize 
the corresponding pairs  $(C_Q,A_Q)$.    This work is currently  under
investigation and will be reported later.

\bibliography{Nudelman}

\end{document}